\documentclass[a4paper,12pt,onecolumn]{article}

\usepackage[vmargin=2cm,hmargin=2cm,headheight=14.5pt,top=2cm,headsep=.5cm]{geometry}
\usepackage{mathptmx}
\usepackage{bm}
\usepackage{stackrel}
\usepackage{amsthm,amsmath,amscd}
\usepackage{amssymb,amsfonts}
\usepackage{makeidx}
\usepackage[all]{xy}
\usepackage{mathptmx}
\usepackage{perpage}
\usepackage[symbol]{footmisc}

\MakePerPage[2]{footnote}
\usepackage{graphicx}
\DeclareMathSizes{12}{12}{8}{6}

\usepackage{cite}

\newtheoremstyle{ptheorem}{1em}{0em}{\itshape}{}{\bfseries}{.}{.5em}{}
\theoremstyle{ptheorem}
\newtheorem{thm}{Theorem}[section]
\newtheorem{pro}[thm]{Proposition}
\newtheorem{lem}[thm]{Lemma}
\newtheorem{cor}[thm]{Corollary}

\theoremstyle{definition}
\newtheorem{dfn}{Definition}[section]

\theoremstyle{remark}
\newtheorem{exa}{Example}[section]

\newtheorem{rem}{Remark}[section]

\numberwithin{equation}{section}
\numberwithin{figure}{section}

\DeclareMathOperator{\sign}{sign}

\DeclareMathOperator{\dif}{d}

\newcommand{\cC}{{\mathcal C}}

\newcommand{\cF}{{\mathcal F}}

\newcommand{\cM}{{\mathcal M}}

\newcommand{\bC}{{\mathbb C}}

\newcommand{\bR}{{\mathbb R}}

\newcommand{\bZ}{{\mathbb Z}}

\renewcommand{\a}{\alpha}
\renewcommand{\b}{\beta}
\renewcommand{\c}{\gamma}
\renewcommand{\l}{\lambda}
\newcommand{\e}{\epsilon}
\renewcommand{\L}{\Lambda}

\renewcommand{\phi}{\varphi}

\newcommand{\ol}{\overline}
\newcommand{\fa}{\forall}

\newcommand{\nkp}{\enskip}

\newcommand{\sfa}{\nkp\fa}
\renewcommand{\d}{\delta}

\renewcommand{\(}{\left(}
\renewcommand{\)}{\right)}
\renewcommand{\[}{\left[}
\renewcommand{\]}{\right]}

\begin{document}
\title{Existence results for a linear equation with reflection, non-constant coefficient and periodic boundary conditions\footnote{Partially supported by FEDER and Ministerio de Educaci\'on y Ciencia, Spain, project MTM2010-15314}}

\author{
Alberto Cabada \, and F. Adri\'an F. Tojo\footnote{Supported by  FPU scholarship, Ministerio de Educaci\'on, Cultura y Deporte, Spain.} \\
\normalsize
Departamento de An\'alise Ma\-te\-m\'a\-ti\-ca, Facultade de Matem\'aticas,\\ 
\normalsize Universidade de Santiago de Com\-pos\-te\-la, Spain.\\ 
\normalsize e-mail: alberto.cabada@usc.es, fernandoadrian.fernandez@usc.es}
\date{}

\maketitle

\begin{abstract}
This work is devoted to the study of first order linear problems with involution and periodic boundary value conditions. We first prove a correspondence between a large set of such problems with different involutions to later focus our attention to the case of the reflection. We study then different cases for which a Green's function can be obtained explicitly and derive several results in order to obtain information about its sign. Once the sign is known, maximum and anti-maximum principles follow. We end this work with more general existence and uniqueness of solution results.
\end{abstract}


\section{Introduction}
In a previous paper by the authors \cite{Cab4}, a Green's function for the following linear problem with reflection was found.
\begin{equation}\label{eqoriginal}x'(t)+\omega x(-t)=h(t), t\in I;\quad x(-T)=x(T),\end{equation}
where $T\in\bR^+$, $\omega\in\bR\backslash\{0\}$ and $h\in L^1(I)$, with $I=[-T,T]$. The precise form of this Green's function was given by the following theorem.
\begin{thm}[\cite{Cab4}, Proposition 3.2]\label{Greenf} Suppose that $\omega \neq k \, \pi/T$, $k \in \bZ$. Then problem (\ref{eqoriginal}) has a unique solution given by the expression
\begin{equation}
\label{e-u}
u(t):=\int_{-T}^T\ol G(t,s)h(s)\dif s,
\end{equation}
where $$\ol{G}(t,s):=\omega\,G(t,-s)-\frac{\partial G}{\partial s}(t,s)$$
and $G$ is the Green's function for the harmonic oscillator
$$x''(t)+\omega^2x(t)=0;\ x(T)=x(-T),\ x'(T)=x'(-T).$$
\end{thm}
The sign properties of this Green's function were further studied in \cite{Cab5}, where the methods similar to those found in \cite{gijwjiea, gijwjmaa, gijwems} are used in order to derive existence and multiplicity results.\par
Still, obtaining the Green's function of problem \eqref{eqoriginal} with a non-constant coefficient has not been accomplished yet. In this article we will study this case and further generalize the existence of Green's functions. Through a correspondence theorem, we will also be able to extend these results to problems with other involutions. We will also obtain new maximum and anti-maximum principles and existence and uniqueness results.

\section{Order one linear problems with involutions}
Assume $\phi$ is a differentiable involution on $[\phi(T),T]$. Let $a,b,c,d\in L^1([\phi(T),T])$ and consider the following problem
\begin{equation}\label{proinv1} d(t)x'(t)+c(t)x'(\phi(t))+b(t)x(t)+a(t)x(\phi(t))=h(t),\ x(\phi(T))=x(T).
\end{equation}
It would be interesting to know under what circumstances problem \eqref{proinv1} is equivalent to another problem of the same kind but with a different involution, in particular the reflection. The following two results will help us to clarify this situation.
\begin{lem}[\textsc{Correspondence of Involutions}]\label{corofinv} Let $\phi$ and $\psi$ be two differentiable involutions\footnote{Every differentiable involution is a diffeomorphism.} on the intervals $[\phi(T),T]$ and $[\psi(S),S]$ respectively. Let $t_0$ and $s_0$ be the unique fixed points of $\phi$ and $\psi$ respectively. Then, there exists an orientation preserving diffeomorphism $f:[\psi(S),S]\to[\phi(T),T]$ such that  $f(\psi(s))=\phi(f(s))\sfa s\in[\psi(S),S]$.
\end{lem}
\begin{proof}
Let $g:[\psi(S),s_0]\to[\phi(T),t_0]$ be an orientation preserving diffeomorphism, that is, $g(s_0)=t_0$. Let us define
$$f(s):=\begin{cases} g(s) & \text{ if } s\in[\psi(S),s_0], \\ (\phi\circ g\circ\psi)(s) & \text{ if } s\in(s_0,S].\end{cases}$$
\par
Clearly,  $f(\psi(s))=\phi(f(s))\sfa s\in[\psi(S),S]$. Since $s_0$ is a fixed point for $\psi$, $f$ is continuous. Furthermore, because $\phi$ and $\psi$ are involutions, $\phi'(t_0)=\psi'(s_0)=-1$, so $f$ is differentiable. $f$ is invertible with inverse
$$f^{-1}(t):=\begin{cases} g^{-1}(t) & \text{ if } t\in[\phi(T),t_0], \\ (\psi\circ g^{-1}\circ\phi)(t) & \text{ if } t\in(t_0,T].\end{cases}$$
$f^{-1}$ is also differentiable for the same reasons.
\end{proof}
\begin{rem}A similar argument could be done in the case of involutions defined on open, possibly not bounded, intervals.
\end{rem}
\begin{rem} The expression obtained for $f$ reminds us of the characterization of involutions given in \cite[Property 6]{Wie}.
\end{rem}
\begin{rem} It is easy to check that if $\phi$ is an involution defined on $\bR$ with fixed point $t_0$ then $\psi(t):=\phi(t+t_0-s_0)-t_0+s_0$ is an involution defined on $\bR$ with fixed point $s_0$ (cf. \cite[Property 2]{Wie}). For this particular choice of $\phi$ and $\psi$, we can take $g(s)=s-s_0+t_0$ in Lemma \ref{corofinv} and, in such a case, $f(s)=s-s_0+t_0$ for all $s\in\bR$.
\end{rem}

\begin{cor}[\textsc{Change of Involution}] Under the hypothesis of Lemma \ref{corofinv},
problem \eqref{proinv1} is equivalent to
\begin{equation}\label{proinv2} \frac{d(f(s))}{f'(s)}y'(s)+\frac{c(f(s))}{f'(\psi(s))}y'(\psi(s))+b(f(s))y(s)+a(f(s))y(\psi(s))=h(f(s)),\ y(\psi(S))=y(S).
\end{equation}
\end{cor}
\begin{proof}Consider the change of variable $t=f(s)$ and $y(s):=x(t)=x(f(s))$. Then, using Lemma \ref{corofinv}, it is clear that
$$\frac{\dif y}{\dif s}(s)=\frac{\dif x}{\dif t}(f(s))\frac{\dif f}{\dif s}(s)\quad\text{and}\quad \frac{\dif y}{\dif s}(\psi(s))=\frac{\dif x}{\dif t}(\phi(f(s)))\frac{\dif f}{\dif s}(\psi(s)).$$
Making the proper substitutions in problem \eqref{proinv1} we get problem  \eqref{proinv2} and vice-versa.
\end{proof}
This last results allows us to restrict our study of problem \eqref{proinv1} to the case where $\phi$ is the reflection $\phi(t)=-t$. In the following section we will further restrict our assumptions to the case where $c\equiv0$ in problem \eqref{proinv1}. A comment on how to proceed without this assumption will be done in the Appendix at the end of this work.

\section{Study of the homogeneous equation}
In this section we will study some different cases for the homogeneous equation
\begin{equation}\label{gen-eq} x'(t)+a(t)x(-t)+b(t)x(t)=0,\ t\in I,\end{equation}
where $a,b\in L^1(I)$. In order to solve it, we can consider the decomposition of equation \eqref{gen-eq} used in \cite{Cab4}. For any given function $f$, let $f_e(x):=\frac{f(x)+f(-x)}{2}$ be its even part and $f_o(x):=\frac{f(x)-f(-x)}{2}$ its odd part. Then, the solutions of equation \eqref{gen-eq} satisfy
\begin{align}\label{eq3.1}\begin{pmatrix}x_o' \\ x_e'\end{pmatrix} & =\begin{pmatrix}a_o-b_o & -a_e-b_e \\ a_e-b_e & -a_o-b_o\end{pmatrix}\begin{pmatrix}x_o \\ x_e\end{pmatrix}.
\end{align}
Realize that, a priori, solutions of system \eqref{eq3.1} need not to be pairs of even and  odd functions, nor provide solutions of \eqref{gen-eq}.\par
In order to solve this system, we will restrict problem \eqref{eq3.1} to those cases where the matrix
$$M(t)=\begin{pmatrix}a_o-b_o & -a_e-b_e \\ a_e-b_e & -a_o-b_o\end{pmatrix}(t)$$
satisfies that $[M(t),M(s)]:=M(t)M(s)-M(s)M(t)=0\sfa t,s\in I$, for in that case the solution of the system \eqref{eq3.1} is given by the exponential of the integral of $M$\footnote{See the Appendix for more details on this matter.}. Clearly,
$$[M(t),M(s)]=2 \begin{pmatrix} a_e(t)b_e(s)-a_e(s)b_e(t) & a_o(s)[a_e(t)+b_e(t)]-a_o(t)[a_e(s)+b_e(s)]\\a_o(t)[a_e(s)+b_e(s)]-a_o(s)[a_e(t)+b_e(t)] & a_e(s)b_e(t)-a_e(t)b_e(s)\end{pmatrix}.$$
Let $A(t):=\int_0^t a(s)\dif s$, $B(t):=\int_0^t b(s)\dif s$. Let $\ol M$ be a primitive (save possibly a constant matrix) of $M$. We study now the different cases where $[M(t),M(s)]=0\sfa t,s\in I$. We will always assume $a\not\equiv0$, since the case $a\equiv0$ is the well-known case of an ODE.\par
\textbf{(C1). $b_e=k\,a,\ k\in\bR,\ |k|<1$.}
\par
In this case, $a_o=0$ and $\ol M$ has the form
$$\ol M=\begin{pmatrix}B_e &  -(1+k)A_o\\ (1-k)A_o & -B_e\end{pmatrix}.$$
If we compute the exponential (see note in the Appendix for more information) we get
$$e^{\ol M(t)}=e^{-B_e(t)}\begin{pmatrix}\cos\(\sqrt{1-k^2}A(t)\) & -\frac{1+k}{\sqrt{1-k^2}}\sin\(\sqrt{1-k^2}A(t)\)\\ \frac{\sqrt{1-k^2}}{1+k}\sin\(\sqrt{1-k^2}A(t)\) & \cos\(\sqrt{1-k^2}A(t)\)\end{pmatrix}.$$
Therefore, if a solution to equation \eqref{gen-eq} exists, it has to be of the form
$$u(t)=\a e^{-B_e(t)}\cos\(\sqrt{1-k^2}A(t)\)+\b e^{-B_e(t)}\frac{1+k}{\sqrt{1-k^2}}\sin\(\sqrt{1-k^2}A(t)\).$$
with $\a$, $\b\in\bR$. It is easy to check that all the solutions of equation \eqref{gen-eq} are of this form with $\b=-\a$.
\par
\textbf{(C2). $b_e=k\,a,\ k\in\bR,\ |k|>1$.}
This case is much similar to (C1) and it yields solutions of system \eqref{eq3.1} of the form
$$u(t)=\a e^{-B_e(t)}\cosh\(\sqrt{k^2-1}A(t)\)+\b e^{-B_e(t)}\frac{1+k}{\sqrt{k^2-1}}\sinh\(\sqrt{k^2-1}A(t)\),$$
which are solutions of equation \eqref{gen-eq} when $\b=-\a$.\par
\textbf{(C3). $b_e=a$.}
In this case the solutions of system \eqref{eq3.1} are of the form
\begin{equation}\label{eqc3}u(t)=\a e^{-B_e(t)}+2\b e^{-B_e(t)}A(t)\end{equation}
which are solutions of equation \eqref{gen-eq} when $\b=-\a$.\par
\textbf{(C4). $b_e=-a$.}
In this case the solutions of system \eqref{eq3.1} are the same as in case (C3), but they are solutions of equation \eqref{gen-eq} when $\b=0$.\par
\textbf{(C5). $b_e=a_e=0$.}
In this case the solutions of system \eqref{eq3.1} are of the form
$$u(t)=\a e^{A(t)-B(t)}+\b e^{-A(t)-B(t)},$$
which are solutions of equation \eqref{gen-eq} when $\a=0$.\par

\section{The cases (C1)--(C3) for the complete problem}

In the more complicated setting of the following nonhomogeneous problem
\begin{equation}\label{eq2cp} x'(t)+a(t)\,x(-t)+b(t)\,x(t)=h(t),\nkp a.\,e. t\in I,\quad x(-T)=x(T),
\end{equation}
we have still that, in the cases (C1)--(C3), it can be sorted out very easily. In fact, we get the expression of the Green's function for the operator.
We remark that in the three considered cases along this section the function $a$ must be even on $I$. We note also that $a$ is allowed to change its sign on $I$. 

\par
First, we are going to prove a generalization of Theorem \ref{Greenf}.\par
Consider problem \eqref{eq2cp} with $a$ and $b$ constants.
\begin{equation}\label{eq2cp2} x'(t)+a\,x(-t)+b\,x(t)=h(t),\nkp t\in I,\quad x(-T)=x(T).
\end{equation}
Considering the homogeneous case ($h=0$), differentiating and making proper substitutions, we arrive to the problem.
\begin{equation}\label{eqhog} x''(t)+(a^2-b^2)x(t)=0,\nkp t\in I,\quad x(-T)=x(T),\quad x'(-T)=x'(T).
\end{equation}
Which, for $b^2<a^2$, is the problem of the harmonic oscillator. It was shown in \cite[Proposition 3.1]{Cab4} that, under uniqueness conditions, the Green's function $G$ for problem \eqref{eqhog} satisfies the following properties in the case $b^2<a^2$, but they can be extended almost automatically to the case $b^2>a^2$.
\begin{lem} The Green's function $G$ satisfies the following properties.
\begin{enumerate}
\item $G\in\cC(I^2,\bR)$,
\item $\frac{\partial G}{\partial t}$ and $\frac{\partial^2 G}{\partial t^2}$ exist and are continuous in $\{(t,s)\in I^2\ |\ s\ne t\}$,
\item $\frac{\partial G}{\partial t}(t,t^-)$ and $\frac{\partial G}{\partial t}(t,t^+)$ exist for all $t\in I$ and satisfy
$$\frac{\partial G}{\partial t}(t,t^-)-\frac{\partial G}{\partial t}(t,t^+)=1\sfa t\in I,$$
\item $\frac{\partial^2 G}{\partial t^2}+(a^2-b^2)G=0\text{ in }\{(t,s)\in I^2\ |\ s\ne t\},$
\item \begin{enumerate}
\item $G(T,s)=G(-T,s)\sfa s\in I$,
\item $\frac{\partial G}{\partial t}(T,s)=\frac{\partial G}{\partial t}(-T,s)\sfa s\in(-T,T)$.
\end{enumerate}
\item $G(t,s)=G(s,t)$,
\item $G(t,s)=G(-t,-s)$,
\item $\frac{\partial G}{\partial t}(t,s)=\frac{\partial G}{\partial s}(s,t)$,
\item $\frac{\partial G}{\partial t}(t,s)=-\frac{\partial G}{\partial t}(-t,-s)$,
\item $\frac{\partial G}{\partial t}(t,s)=-\frac{\partial G}{\partial s}(t,s)$.
\end{enumerate}
\end{lem}
With these properties, we can prove the following Theorem (cf. \cite[Proposition 3.2]{Cab4}).
\begin{thm}\label{Greenf2} Suppose that $a^2-b^2 \neq n^2 \, (\pi/T)^2$, $n=0,1,\dots$ Then problem \eqref{eq2cp2} has a unique solution given by the expression
\begin{equation*}
\label{e-u2}
u(t):=\int_{-T}^T\ol G(t,s)h(s)\dif s,
\end{equation*}
where \begin{equation}\label{e-a-b}
\ol{G}(t,s):=a\,G(t,-s)-b\,G(t,s)+\frac{\partial G}{\partial t}(t,s)\end{equation}
is called the \textbf{Green's function} related to problem \eqref{eq2cp2}.
\end{thm}
\begin{proof}
Since problem \eqref{eq2cp2}, in the homogeneous case, can be reduced to a problem with the equation of problem \eqref{eqhog}, the classical theory of ODE tells us that problem \eqref{eq2cp2} has at most one solution for all $a^2-b^2 \neq n^2 \, (\pi/T)^2$, $n=0,1,\dots$ Let us see that function $u$ defined in (\ref{e-u}), with $\ol G$ given by \eqref{e-a-b}, fulfills (\ref{eq2cp2}):

\begin{eqnarray*}
& & u'(t)+a\, u(-t)+b\, u(t)=\frac{\dif}{\dif t}\int_{-T}^{-t}\ol G(t,s)h(s)\dif s+\frac{\dif}{\dif t}\int_{-t}^t\ol G(t,s)h(s)\dif s+\frac{\dif}{\dif t}\int_{t}^T\ol G(t,s)h(s)\dif s\\
&+&a\int_{-T}^T\ol G(-t,s)h(s)\dif s+b\int_{-T}^T\ol G(t,s)h(s)\dif s\\
&=&(\ol G(t,t^-)-\ol G(t,t^+))h(t)+\int_{-T}^T\left[a\frac{\partial G}{\partial t}(t,-s)-b\frac{\partial G}{\partial t}(t,s)+\frac{\partial^2 G}{\partial t^2}(t,s)\right]h(s)\dif s\\
&+&a\int_{-T}^T\left[a\,G(-t,-s)-b\,G(-t,s)+\frac{\partial G}{\partial t}(-t,s)\]h(s)\dif s
+b\int_{-T}^T\left[a\,G(t,-s)-b\,G(t,s)+\frac{\partial G}{\partial t}(t,s)\]h(s)\dif s.
\end{eqnarray*}

Using properties $(I)-(X)$, we deduce that this last expression is equal to $h(t)$, so the equation in problem \eqref{eq2cp2} is satisfied. 

Property $(V)$ allows us to verify the boundary conditions.
$$ u(T)-u(-T)=$$
$$\int_{-T}^T\left[a\,G(T,-s)-b\,G(T,s)+\frac{\partial G}{\partial t}(T,s)-a\,G(-T,-s)+b\,G(-T,s)-\frac{\partial G}{\partial t}(-T,s)\right]h(s)\dif s=0.$$
\end{proof}
This last theorem leads us to the question ``Which is the Green's function for the case (C3) with $a,b$ constants?". The following Lemma answers that question.
\begin{lem}\label{lemGc3}Let $a\ne 0$ be a constant and let $G_{C3}$ be a real function defined as
$$G_{C3}(t,s):=\frac{t-s}{2}-a\,s\,t+\begin{cases} -\frac{1}{2}+a\,s & \text{ if } |s|<t, \\ \frac{1}{2}-a\,s & \text{ if } |s|<-t, \\ \frac{1}{2}+a\,t & \text{ if } |t|<s, \\ -\frac{1}{2}-a\,t & \text{ if } |t|<-s.\end{cases}$$
Then the following properties hold.
\begin{itemize}
\item $\frac{\partial G_{C3}}{\partial t}(t,s)+a(G_{C3}(t,s)+G_{C3}(-t,s))=0$ for a.\,e. $t,s\in (-1,1)$.
\item $\frac{\partial G_{C3}}{\partial t}(t,t^+)-\frac{\partial G_{C3}}{\partial t}(t,t^-)=1\sfa t\in(-1,1)$.
\item $G_{C3}(-1,s)=G_{C3}(1,s)\sfa s\in(-1,1)$.
\end{itemize}
\end{lem}
These properties are straightforward to check.\ Clearly, $G_{C3}$ is the Green's function for the problem
$$x'(t)+a[x(t)+x(-t)]=h(t), t\in[-1,1];\quad x(1)=x(-1),$$
that is, the Green's function for the case (C3) with $a,b$ constants and $T=1$. For other values of $T$, it is enough to make a change of variables.
\begin{rem} The function $G_{C3}$ can be obtained from the Green's functions for the case $(C1)$ with $a$ constant, $b_o\equiv0$ and $T=1$ taking the limit $k\to 1^-$ for $T=1$.
\end{rem}

The following theorem shows how to obtain a Green's function for non constant coefficients of the equation using the Green's function for constant coefficients. We can find the same principle, that is, to compose a Green's function with some other function in order to obtain a new Green's function, in \cite[Theorem 5.1, Remark 5.1]{Cab6} and also in \cite[Section 2]{Gau}.\par
But first, we need to now hot the Green's function should be defined in such a case. Theorem \ref{Greenf2} gives us the expression of the Green's function for problem \eqref{eq2cp2}, $\ol{G}(t,s):=a\,G(t,-s)-b\,G(t,s)+\frac{\partial G}{\partial t}(t,s)$. For instance, in the case (C1), if $\omega=\sqrt{a^2-b^2}$,
$$2\omega\sin(\omega T)\ol{G}(t,s):=\begin{cases} a \cos[\omega (s + t - T)]+ b \cos[\omega (s - t + T)] + \omega \sin[\omega (s - t + T)],& t>|s|,\\
a\cos[\omega (s + t - T)] +b \cos[\omega (-s + t + T)] - \omega \sin[\omega (-s + t + T)],   & s>|t|,\\
a \cos[\omega (s + t + T)] +b \cos[\omega (-s + t + T)] - \omega \sin[\omega (-s + t + T)], & -t>|s|,\\
a \cos[\omega (s + t + T)] +b \cos[\omega (s - t + T)] + \omega \sin[\omega (s - t + T)],  & -s>|t|. \end{cases}$$
Also, observe that $\ol G$ is continuous except at the diagonal, where $\ol G(t,t^-)-\ol G(t,t^+)=1$.

Similarly, we can obtain the explicit expression of the Green's function $\ol G$ for the cases (C2) and (C3) (see Lemma \ref{lemGc3}). In any case, we have that the Green's function for problem \eqref{eq2cp2} can be expressed as
$$2\omega\sin(\omega T)\ol{G}(t,s):=\begin{cases} \ol G_1(t,s),& t>|s|,\\
\ol G_2(t,s),   & s>|t|,\\
\ol G_3(t,s), & -t>|s|,\\
\ol G_4(t,s),  & -s>|t|, \end{cases}$$
were the $\ol G_j$, $j=1,\dots,4$ are analytic functions defined on $\bR^2$. 

In order to simplify the statement of the following Theorem, consider the following conditions.\par
$\mathbf{(C1^*)}$. (C1) is satisfied, $(1-k^2)A(T)^2\neq (n \, \pi)^2$ for all $n=0,1,\dots$ and $\cos\(\sqrt{1-k^2}A(T)\)\ne0$.\par
$\mathbf{(C2^*)}$. (C2) is satisfied and $(1-k^2)A(T)^2\neq (n \, \pi)^2$  for all $n=0,1,\dots$\par
$\mathbf{(C3^*)}$. (C3) is satisfied and $A(T)\ne0$.\par
Assume one of $(C1^*)$--$(C3^*)$. In that case, by Theorem \ref{Greenf2} and Lemma \ref{lemGc3}, we are under uniqueness conditions for the solution for the following problem \cite{Cab4}.
\begin{equation}\label{eq2} x'(t)+x(-t)+k\,x(t)=h(t),\nkp t\in [-|A(T)|,|A(T)|],\quad  x(A(T))=x(-A(T)).
\end{equation}

The Green's function $G_2$ for problem \eqref{eq2} is just an specific case of $\ol G$ and can be expressed as
$$\ol{G_2}(t,s):=\begin{cases} k_1(t,s),& t>|s|,\\
k_2(t,s),   & s>|t|,\\
k_3(t,s), & -t>|s|,\\
k_4(t,s),  & -s>|t|. \end{cases}$$
Define now 
\begin{equation}\label{Ggenral}G_1(t,s):=e^{B_e(s)-B_e(t)}H(t,s)=e^{B_e(s)-B_e(t)}\begin{cases} k_1(A(t),A(s)),& t>|s|,\\
k_2(A(t),A(s)),   & s>|t|,\\
k_3(A(t),A(s)), & -t>|s|,\\
k_4(A(t),A(s)),  & -s>|t|. \end{cases}\end{equation}
Defined this way, $G_1$ is continuous except at the diagonal, where $G_1(t,t^-)-\ol G_1(t,t^+)=1$. Now we can state the following Theorem.
\begin{thm}\label{thmcases123}
Assume one of $(C1^*)$--$(C2^*)$. Let $G_1$ be defined as in \eqref{Ggenral}. Assume $G_1(t,\cdot)h(\cdot)\in L^1(I)$ for every $t\in I$. Then problem \eqref{eq2cp} has a unique solution given by
$$u(t)=\int_{-T}^TG_1(t,s)h(s)\dif s.$$
\end{thm}
\begin{proof}
First realize that, since $a$ is even, $A$ is odd, so $A(-t)=-A(t)$. It is important to note that if $a$ has not constant sign in $I$, then $A$ may be not injective on $I$.

From the properties of $\bar G_2$ as a Green's function, it is clear that
$$\frac{\partial \bar G_2}{\partial t}(t,s)+\bar G_2(-t,s)+k\,\bar G_2(t,s)=0\quad\text{for a.\,e. }t,s\in A(I),$$
and so,
$$\frac{\partial H}{\partial t}(t,s)+a(t)H(-t,s)+ka(t)\,H(t,s)=0\quad\text{for a.\,e. }t,s\in I,$$
Hence 
\begin{align*}
& u'  (t)+a(t)\, u(-t)+(b_o(t)+k\,a(t))\, u(t)= \frac{\dif}{\dif t}\int_{-T}^TG_1(t,s)h(s)\dif s+a(t)\int_{-T}^TG_1(-t,s)h(s)\dif s\\
&\quad+(b_o(t)+k\,a(t))\int_{-T}^TG_1(t,s)h(s)\dif s\\
=\ &\frac{\dif}{\dif t}\int_{-T}^{t}  e^{B_e(s)-B_e(t)}H(t,s)h(s)\dif s+\frac{\dif}{\dif t}\int_{t}^T e^{B_e(s)-B_e(t)} H(t,s)h(s)\dif s\\&\quad+ a(t)\int_{-T}^T e^{B_e(s)-B_e(t)}H(-t,s)h(s)\dif s+(b_o(t)+k\,a(t))\int_{-T}^T e^{B_e(s)-B_e(t)}H(t,s)h(s)\dif s\\
=\ &[H(t,t^-)-H(t,t^+)]h(t)+a(t)\, e^{-B_e(t)}\int_{-T}^Te^{B_e(s)}\frac{\partial H}{\partial t}(t,s)h(s)\dif s\\&\quad- b_o(t)e^{-B_e(t)}\int_{-T}^Te^{B_e(s)}H(t,s)h(s)\dif s+ a(t)e^{-B_e(t)}\int_{-T}^Te^{B_e(s)} H(-t,s)h(s)\dif s\\ &\quad+ (b_o(t)+k\,a(t))e^{-B_e(t)}\int_{-T}^Te^{B_e(s)}H(t,s)h(s)\dif s\\ =\ & h(t)+\, a(t)e^{-B_e(t)}\int_{-T}^Te^{B_e(s)}\[\frac{\partial H}{\partial t}(t,s)+a(t)H(-t,s)+ka(t)\,H(t,s)\]h(s)\dif s=h(t).
\end{align*}

The boundary conditions are also satisfied.
$$u(T)-u(-T)= e^{-B_e(T)}\int_{-T}^Te^{B_e(s)} [H(T,s)-H(-T,s)]h(s)\dif s=0.$$
In order to check the uniqueness of solution, let $u$ and $v$ be solutions of problem \eqref{eq2}. Then $u-v$ satisfies equation \eqref{gen-eq} and so is of the form given when we first studied the cases $(C1^*)$--$(C3^*)$ (see Section 3). Also, $(u-v)(T)-(u-v)(-T)=2(u-v)_o(T)=0$, but this can only happen, by what has been imposed by conditions $(C1^*)$--$(C3^*)$, if $u-v\equiv0$, thus proving the uniqueness of solution.
\end{proof}

\begin{exa}
Consider the problem
$$x'(t)=\cos(\pi t)x(-t)+\sinh(t)x(t)=\cos(\pi t)+\sinh(t), \;x(3/2)=x(-3/2). $$
Clearly we are in the case (C1). If we compute the Green's function according to Theorem \ref{thmcases123} we obtain

 $$2\sin(\sin (\pi T))G_1(t,s)=e^{\cosh (s)-\cosh (t)}\begin{cases}
  \sin \left(\frac{\sin (\pi s)}{\pi }-\frac{\sin (\pi t)}{\pi }-\frac{\sin (\pi T)}{\pi }\right)+ \cos \left(\frac{\sin (\pi s)}{\pi }+\frac{\sin (\pi t)}{\pi }-\frac{\sin (\pi T)}{\pi }\right), |t|<s,\\
    \sin \left(\frac{\sin (\pi s)}{\pi}-\frac{\sin (\pi t)}{\pi }+\frac{\sin (\pi T)}{\pi }\right)+\cos \left(\frac{\sin (\pi s)}{\pi }+\frac{\sin (\pi t)}{\pi }+\frac{\sin (\pi T)}{\pi }\right), |t|<-s,\\
  \sin \left(\frac{\sin (\pi s)}{\pi}-\frac{\sin (\pi t)}{\pi }+\frac{\sin (\pi T)}{\pi }\right)+ \cos \left(\frac{\sin (\pi s)}{\pi }+\frac{\sin (\pi t)}{\pi }-\frac{\sin (\pi T)}{\pi }\right), |s|<t,\\
  \sin \left(\frac{\sin (\pi s)}{\pi }-\frac{\sin (\pi t)}{\pi }-\frac{\sin (\pi T)}{\pi }\right) + \cos \left(\frac{\sin (\pi s)}{\pi }+\frac{\sin (\pi t)}{\pi }+\frac{\sin (\pi T)}{\pi }\right), |s|<-t.\end{cases}$$
  \begin{figure}[hhht]\label{figure2g}
  \center{\includegraphics[width=.5\textwidth]{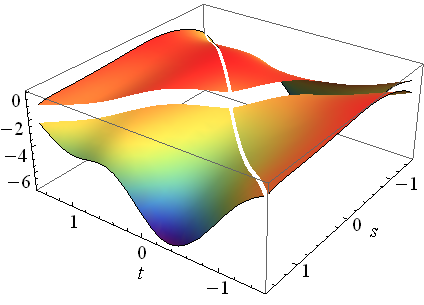}\includegraphics[width=.5\textwidth]{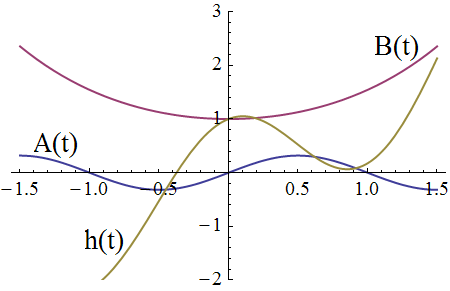}}\caption{Graphs of the kernel \textit{(left)} and of the functions involved in the problem \textit{(right)}.}
  \end{figure}
\end{exa}
One of the most important direct consequences of Theorem \ref{thmcases123} is the existence of maximum and antimaximum principles in the case $b\equiv0$\footnote{Note that this discards the case (C3), for which $b\equiv0$ implies $a\equiv 0$, because we are assuming $a\not\equiv0$.}. To show this and what happens in the case $b$ constant, $b\ne0$, we recall here a couple results from \cite{Cab4}.
\begin{thm}[{\cite[Theorem 4.3]{Cab4}}]\label{alphasign}Let $b=0$, $\a=aT$. \par
\begin{itemize}
\item If $\a\in(0,\frac{\pi}{4})$ then $\ol G$ is strictly positive on $I^2$.
\item If $\a\in(-\frac{\pi}{4},0)$ then $\ol G$ is strictly negative on $I^2$.
\item If $\a=\frac{\pi}{4}$ then $\ol G$ vanishes on $P:=\{(-T,-T),(0,0),(T,T),(T,-T)\}$ and is strictly positive on $(I^2)\backslash P$.
\item If $\a=-\frac{\pi}{4}$ then $\ol G$ vanishes on $P$ and is strictly negative on $(I^2)\backslash P$.
\item If $\a\in\bR\backslash[-\frac{\pi}{4},\frac{\pi}{4}]$ then $\ol G$ is not positive nor negative on $I^2$.
\end{itemize}
\end{thm}
\begin{cor}[{\cite[Corollary 4.4]{Cab4}}]\label{coralphasign} Let $\cF_\l(I)$ be the set of real differentiable functions $f$ defined on $I$ such that $f(-T)-f(T)=\l$. The operator $R_a:\cF_\l(I)\to L^1(I)$ defined as $R_a(x(t))=x'(t)+a\, x(-t)$, with $a\in\bR\backslash\{0\}$, satisfies
\begin{itemize}
\item $R_m$ is strongly inverse positive if and only if $a\in(0,\frac{\pi}{4T}]$ and $\l\ge0$,
\item $R_m$ is strongly inverse negative if and only if $a\in[-\frac{\pi}{4T},0)$ and $\l\ge0$.

\end{itemize}
\end{cor}
With these results  we get the following corollary to Theorem \ref{thmcases123}.
\begin{cor}\label{corsigng}Under the conditions of Theorem \ref{thmcases123}, if $a$ is nonnegative on $I$ and $b=0$,\par
\begin{itemize}
\item If $A(T)\in(0,\frac{\pi}{4})$ then $G_1$ is strictly positive on $I^2$.
\item If $A(T)\in(-\frac{\pi}{4},0)$ then $G_1$ is strictly negative on $I^2$.
\item If $A(T)=\frac{\pi}{4}$ then $G_1$ vanishes on $P:=\{(-A(T),-A(T)),(0,0),(A(T),A(T)),( A(T),- A(T))\}$ and is strictly positive on $(I^2)\backslash P$.
\item If $A(T)=-\frac{\pi}{4}$ then $G_1$ vanishes on $P$ and is strictly negative on $(I^2)\backslash P$.
\item If $A(T)\in\bR\backslash[-\frac{\pi}{4},\frac{\pi}{4}]$ then $G_1$ is not positive nor negative on $I^2$.
\end{itemize}
Furthermore, the operator $R_a:\cF_\l(I)\to L^1(I)$ defined as $R_a(x(t))=x'(t)+a(t)\, x(-t)$ satisfies
\begin{itemize}
\item $R_a$ is strongly inverse positive if and only if $A(T)\in(0,\frac{\pi}{4T}]$ and $\l\ge0$,
\item $R_a$ is strongly inverse negative if and only if $A(T)\in[-\frac{\pi}{4T},0)$ and $\l\ge0$.

\end{itemize}
\end{cor}

The second part of this last corollary, drawn from positivity (or negativity) of the Green's function could have been obtained, as we show below, without having so much knowledge about the Green's function. In order to show this, consider the following proposition in the line of the work of Torres \cite[Theorem 2.1]{Tor}.\par
\begin{pro}\label{proredpro}
Consider the homogeneous initial value problem
\begin{equation}\label{eqhomivp} x'(t)+a(t)\,x(-t)+b(t)\,x(t)=0,\ t\in I;\ x(t_0)=0.\end{equation}
If problem \eqref{eqhomivp} has a unique solution ($x\equiv 0$) on $I$ for all $t_0\in I$ then, if the Green function for \eqref{eq2cp} exists, it has constant sign.\par
What is more, if we further assume $a+b$ has constant sign, the Green's function has the same sign as $a+b$.
\end{pro}
\begin{proof}
Without lost of generality, consider $a$ to be a $2T$-periodic $L^1$ function defined on $\bR$ (the solution of \eqref{eq2cp} will be considered in $I$). Let $G_1$ be the Green's function for problem \eqref{eq2cp}. Since $G_1(T,s)=G_1(-T,s)$ for all $s\in I$, and $G_1$ is continuous except at the diagonal, it is enough to prove that $G_1(t,s)\ne0\sfa t,s\in I$.\par
Assume, on the contrary, that there exists $t_1,s_1\in I$ such that $G_1(t_1,s_1)=0$. Let $g$ be the $2T$-periodic extension of $G_1(\cdot,s_1)$. Let us assume $t_1>s_1$ (the other case would be analogous). Let $f$ be the restriction of $g$ to $(s_1,s_1+2T)$. $f$ is absolutely continuous and satisfies \eqref{eqhomivp} a.e. for $t_0=t_1$, hence, $f\equiv0$. This contradicts the fact of $G_1$ being a Green's function, therefore $G_1$ has constant sign.\par
Realize now that $x\equiv1$ satisfies
$$x'(t)+a(t)x(-t)+b(t)x(t)=a(t)+b(t),\ x(-T)=x(T).$$
Hence, $\int_{-T}^TG_1(t,s)(a(s)+b(s))\dif s=1$ for all $t\in I$. Since both $G_1$ and $a+b$ have constant sign, they have the same sign.
\end{proof}
The following corollaries are an straightforward application of this result to the cases (C1)--(C3) respectively.
\begin{cor}\label{cor1sig} Assume $a$ has constant sign. Under the assumptions of (C1) and Theorem \ref{thmcases123}, $G_1$ has constant sign if
$$|A(T)|< \frac{\arccos(k)}{2\sqrt{1-k^2}}.$$
Furthermore, $\sign(G_1)=\sign(a)$.
\end{cor}
\begin{proof}
The solutions of \eqref{gen-eq} for the case (C1), as seen before, are given by
$$u(t)  =\a e^{-B_e(t)}\[ \cos\(\sqrt{1-k^2}A(t)\)-\frac{1+k}{\sqrt{1-k^2}}\sin\(\sqrt{1-k^2}A(t)\)\].$$
Using a particular case of the phasor addition formula\footnote{$\a\cos \c+\b\sin \c=\sqrt{\a^2+\b^2}\sin(\c+\theta)$, where $\theta\in[-\pi,\pi)$ is the angle such that $\cos\theta=\frac{\b}{\sqrt{\a^2+\b^2}}$, $\sin\theta=\frac{\a}{\sqrt{\a^2+\b^2}}$.},
 $$u(t)=\a e^{-B_e(t)}\sqrt{\frac{2}{1-k}}\sin\(\sqrt{1-k^2}A(t)+\theta\),$$
 where $\theta\in[-\pi,\pi)$ is the angle such that
\begin{equation}\label{sincos}\sin\theta=\sqrt{\frac{1-k}{2}}\quad\text{and}\quad\cos\theta=-\frac{1+k}{\sqrt{1-k^2}}\sqrt{\frac{1-k}{2}}=-\sqrt{\frac{1+k}{2}}.\end{equation}
Observe that this implies that $\theta\in\(\frac{\pi}{2},\pi\)$.\par
In order for the hypothesis of Proposition \ref{proredpro} to be satisfied, it is enough and sufficient to ask for $0\not\in u(I)$ for some $\a\ne0$. Equivalently, that 
$$\sqrt{1-k^2}A(t)+\theta\neq \pi n\sfa n\in\bZ\sfa t\in I,$$
That is,
$$A(t)\neq \frac{\pi n-\theta}{\sqrt{1-k^2}}\sfa n\in\bZ\sfa t\in I.$$
Since $A$ is odd and injective and $\theta\in\(\frac{\pi}{2},\pi\)$, this is equivalent to
\begin{equation}\label{firststimate}|A(T)|< \frac{\pi-\theta}{\sqrt{1-k^2}}.\end{equation}
Now, using the double angle formula for the sine and \eqref{sincos},
$$\frac{1-k}{2}=\sin^2\theta=\frac{1-\cos(2\theta)}{2}\text{,\quad this is,\quad} k=\cos(2\theta),$$
which implies, since $2\theta\in(\pi,2\pi)$,
$$\theta=\pi-\frac{\arccos(k)}{2},$$
where $\arccos$ is defined such that it's image is $[0,\pi)$. Plugging this into inequality \eqref{firststimate} yields
$$|A(T)|<\sigma(k):= \frac{\arccos(k)}{2\sqrt{1-k^2}},\quad k\in(-1,1).$$
\par
Using $|k|<1$, $a+b=(k+1)a+b_o$ and the continuity of $G_1$ with respect to $a$ and $b$, we can prove that the sign of the Green's function is given by Proposition \ref{proredpro}.
\end{proof}
\begin{rem}
In the case $a$ is a constant $\omega$ and $k=0$, $A(I)=[-|\omega|T,|\omega|T]$, and the condition can be written as $|\omega|T<\frac{\pi}{4}$, which is consistent with the results found in \cite{Cab4}.\end{rem}
\begin{rem}
Observe that $\sigma$ is strictly decreasing on $(-1,1)$ and
$$\lim_{k\to-1^+}\sigma(k)=+\infty,\quad\lim_{k\to1^-}\sigma(k)=\frac{1}{2}.$$
\end{rem}
\begin{cor}\label{corC3sign} Under the conditions of (C3) and Theorem \ref{thmcases123}, $G_{1}$ has constant sign if $|A(T)|<\frac{1}{2}$.
\end{cor}
\begin{proof} This corollary is a direct consequence of equation \eqref{eqc3}, Proposition \ref{proredpro} and Corollary \ref{thmcases123}. Observe that the result is consistent with $\sigma(1^-)=\frac{1}{2}$.
\end{proof}
In order to prove the next corollary, we need the following 'hyperbolic version' of the phasor addition formula. It's proof can be done without difficulty.
\begin{lem}\label{hyppha}
Let $\a$, $\b,\c\in\bR$, then
$$\a\cosh \c + \b\sinh\c= \sqrt{|\a^2-\b^2|}\begin{cases}\cosh\(\frac{1}{2}\ln\left|\frac{\a+\b}{\a-\b}\right|+\c\) & \text{if}\quad \a>|\b|,
\\ -\cosh\(\frac{1}{2}\ln\left|\frac{\a+\b}{\a-\b}\right|+\c\) & \text{if}\quad- \a>|\b|,
\\
\sinh\(\frac{1}{2}\ln\left|\frac{\a+\b}{\a-\b}\right|+\c\) & \text{if}\quad \b>|\a|,
\\
-\sinh\(\frac{1}{2}\ln\left|\frac{\a+\b}{\a-\b}\right|+\c\) & \text{if }\quad -\b>|\a|,
\\
\a\,e^\c & \text{if }\quad\a=\b,\\
\a\,e^{-\c} & \text{if }\quad\a=-\b.\\
\end{cases}$$
\end{lem}
\begin{cor}\label{cor2sig} Assume $a$ has constant sign. Under the assumptions of (C2) and Theorem \ref{thmcases123}, $G_1$ has constant sign if $k<-1$ or
$$|A(T)|<-\frac{\ln(k-\sqrt{k^2-1})}{2\sqrt{k^2-1}}.$$
Furthermore, $\sign(G_1)=\sign(k\,a)$.
\end{cor}
\begin{proof}
The solutions of \eqref{gen-eq} for the case (C2), as seen before, are given by
$$u(t)  =\a e^{-B_e(t)}\[\cosh\(\sqrt{k^2-1}A(t)\)- \frac{1+k}{\sqrt{k^2-1}}\sinh\(\sqrt{k^2-1}A(t)\)\].$$
If $k>1$, then $1<\frac{1+k}{\sqrt{k^2-1}}$, so, using Lemma \ref{hyppha},
 $$u(t)=-\a e^{-B_e(t)}\sqrt{\frac{2k}{k-1}}\sinh\(\frac{1}{2}\ln\left|k-\sqrt{k^2-1}\right|+\sqrt{k^2-1}A(t)\),$$
In order for the hypothesis of Proposition \ref{proredpro} to be satisfied, it is enough and sufficient to ask that $0\not\in u(I)$ for some $\a\ne0$. Equivalently, that 
$$\frac{1}{2}\ln(k-\sqrt{k^2-1})+\sqrt{k^2-1}A(t)\neq 0\sfa t\in I,$$
That is,
$$A(t)\neq -\frac{\ln(k-\sqrt{k^2-1})}{2\sqrt{k^2-1}}\sfa t\in I.$$
Since $A$ is odd and injective, this is equivalent to
$$|A(T)|<\sigma(k):=-\frac{\ln(k-\sqrt{k^2-1})}{2\sqrt{k^2-1}},\quad k>1.$$
Now, if $k<-1$, then $\left|\frac{1+k}{\sqrt{k^2-1}}\right|<1$, so using Lemma \ref{hyppha},
 $$u(t)=\a e^{-B_e(t)}\sqrt{\frac{2k}{k-1}}\cosh\(\frac{1}{2}\ln\left|k-\sqrt{k^2-1}\right|+\sqrt{k^2-1}A(t)\)\ne0\quad\text{for all}\quad t\in I,\ \a\ne0,$$
so the hypothesis of Proposition \ref{proredpro} are satisfied.\par
Using $|k|>1$, $a+b=(k^{-1}+1)b_e+b_o$ and the continuity of $G_1$ on $a$ and $b$ we can prove that the sign of the Green's function is given by Proposition \ref{proredpro}.
\end{proof}

\begin{rem}
If we consider $\sigma$ defined piecewise as in Corollaries \ref{cor1sig} and \ref{cor2sig} and continuously continued through $1/2$, we get
$$\sigma(k):=\begin{cases} \frac{\arccos(k)}{2\sqrt{1-k^2}} & \text{ if } k\in(-1,1) \\
\frac{1}{2} & \text{ if } k=1 \\
-\frac{\ln(k-\sqrt{k^2-1})}{2\sqrt{k^2-1}} & \text{ if } k>1\end{cases}$$
This function is not only continuous (it is defined thus), but also analytic. In order to see this it is enough to consider the extended definition of the logarithm and the square root to the complex numbers. Remember that $\sqrt{-1}:=i$ and that the principal branch of the logarithm is defined as $\ln_0(z)=\ln|z|+i\theta$ where $\theta\in[-\pi,\pi)$ and $z=|z|e^{i\theta}$ for all $z\in\bC\backslash\{0\}$. Clearly, $\ln_0|_{(0,+\infty)}=\ln$.\par
Now, for $|k|<1$,
$\ln_0(k-\sqrt{1-k^2}i)=i\theta$ with $\theta\in[-\pi,\pi)$ such that $\cos\theta=k$, $\sin\theta=-\sqrt{1-k^2}$, that is, $\theta\in[-\pi,0]$. Hence, $i\ln_0(k-\sqrt{1-k^2}i)=-\theta\in[0,\pi]$. Since $\cos(-\theta)=k$, $\sin(-\theta)=\sqrt{1-k^2}$, it is clear that
$$\arccos(k)=-\theta=i\ln_0(k-\sqrt{1-k^2}i).$$
We thus extend $\arccos$ to $\bC$ by
$$\arccos(z):=i\ln_0(z-\sqrt{1-z^2}i),$$
which is clearly an analytic function. So, if $k>1$,
$$\sigma(k)=-\frac{\ln(k-\sqrt{k^2-1})}{2\sqrt{k^2-1}}=-\frac{\ln_0(k-i\sqrt{1-k^2})}{2i\sqrt{1-k^2}}=\frac{i\ln_0(k-i\sqrt{1-k^2})}{2\sqrt{1-k^2}}=\frac{\arccos(k)}{2\sqrt{1-k^2}}.$$
$\sigma$ is positive, strictly decreasing and
$$\lim_{k\to -1^+}\sigma(k)=+\infty,\quad\lim_{k\to+\infty}\sigma(k)=0.$$
\end{rem}
In a similar way to Corollaries \ref{cor1sig},\ref{corC3sign} and \ref{cor2sig}, we can prove results not assuming $a$ to be a constant sign function. The result is the following.

\begin{cor}
Under the assumptions of Theorem \ref{thmcases123} and conditions (C1), (C2) or (C3) (let $k$ be the constant involved in such conditions), $G_1$ has constant sign if $\max A(I)<\sigma(k)$.
\end{cor}
\section{The cases (C4) and (C5)}

Consider the following problem derived from the nonhomogeneous problem \eqref{eq2cp}.

\begin{align}\label{eoparts}\begin{pmatrix}x_o' \\ x_e'\end{pmatrix} & =\begin{pmatrix}a_o-b_o & -a_e-b_e \\ a_e-b_e & -a_o-b_o\end{pmatrix}\begin{pmatrix}x_o \\ x_e\end{pmatrix}+\begin{pmatrix}h_e \\ h_o\end{pmatrix}.
\end{align}
The following theorems tell us what happens when we impose the boundary conditions.

\begin{thm} If condition (C4) holds, then problem \eqref{eq2cp} has solution if and only if
$$\int_0^Te^{B_e(s)}h_e(s)\dif s=0,$$
and in that case the solutions of \eqref{eq2cp} are given by
\begin{equation}
\label{e-c4}
u_c(t)=e^{-B_e(t)}\[c+\int_0^t\(e^{B_e(s)}h(s)+2a_e(s)\int_0^se^{B_e(r)}h_e(r)\dif r\)\dif s\]\nkp\text{for } c\in\bR.\end{equation}
\end{thm}
\begin{proof} We know that any solution of problem \eqref{eq2cp} has to satisfy \eqref{eoparts}. In the case (C4), the matrix in \eqref{eoparts} is lower triangular
\begin{align}\label{eoparts3}\begin{pmatrix}x_o' \\ x_e'\end{pmatrix} & =\begin{pmatrix}-b_o & 0 \\ 2a_e & -b_o\end{pmatrix}\begin{pmatrix}x_o \\ x_e\end{pmatrix}+\begin{pmatrix}h_e \\ h_o\end{pmatrix}.
\end{align}
 so, the solutions of \eqref{eoparts3} are given by
 
\begin{align*}x_o(t) &
=e^{-B_e(t)}\[\tilde c+\int_0^te^{B_e(s)}h_e(s)\dif s\]
 ,\\x_e(t) & =e^{-B_e(t)}\[c+\int_0^t\(e^{B_e(s)}h_o(s)+2a_e(s)\[\tilde c+\int_0^se^{B_e(r)}h_e(r)\dif r\]\)\dif s\],\end{align*}
where $c$, $\tilde c\in\bR$. $x_e$ is even independently of the value of $c$. Nevertheless, $x_o$ is odd only when $\tilde c=0$. Hence, a solution of \eqref{eq2cp}, if it exists, it has the form
\eqref{e-c4}.\par
To show the second implication it is enough to check that $u_c$ is a solution of the problem \eqref{eq2cp}.
\begin{align*}
u'_c(t)  = & -b_o(t)e^{-B_e(t)}\[c+\int_0^t\(e^{B_e(s)}h(s)+2a_e(s)\int_0^se^{B_e(r)}h_e(r)\dif r\)\dif s\] \\& +e^{-B_e(t)}\(e^{B_e(t)}h(t)+2a_e(t)\int_0^te^{B_e(r)}h_e(r)\dif r\)=h(t)-b_o(t)u(t)+2a_e(t)e^{-B_e(t)}\int_0^te^{B_e(r)}h_e(r)\dif r.
\end{align*}
Now,
\begin{align*}
& a_e(t)(u_c(-t)-u_c(t))+2a_e(t)e^{-B_e(t)}\int_0^te^{B_e(r)}h_e(r)\dif r\\= & a_e(t)e^{-B_e(t)}\[c-\int_0^t\(e^{B_e(s)}h(-s)-2a_e(s)\int_0^se^{B_e(r)}h_e(r)\dif r\)\dif s\]\\ & -  a_e(t)e^{-B_e(t)}\[c+\int_0^t\(e^{B_e(s)}h(s)+2a_e(s)\int_0^se^{B_e(r)}h_e(r)\dif r\)\dif s\]+2a_e(t)e^{-B_e(t)}\int_0^te^{B_e(r)}h_e(r)\dif r\\= & -2a_e(t)e^{-B_e(t)}\int_0^te^{B_e(r)}h_e(r)\dif s+2a_e(t)e^{-B_e(t)}\int_0^te^{B_e(r)}h_e(r)\dif r=  0.
\end{align*}
Hence,
$$u_c'(t)+a_e(t)u_c(-t)+(-a_e(t)+b_o(t))u_c(t)=h(t),\ a.\,e. t\in I.$$
The boundary condition $x(-T)-x(T)=0$ is equivalent to $x_o(T)=0$, this is,
$$\int_0^Te^{B_e(s)}h_e(s)\dif s=0$$
and the result is concluded.
\end{proof}

\begin{thm} If condition (C5) holds, then problem \eqref{eq2cp} has solution if and only if
\begin{equation}
\label{conodd2}
\int_0^T e^{B(s)-A(s)}h_e(s)\dif s=0,\end{equation}
and in that case the solutions of \eqref{eq2cp} are given by
\begin{equation}
\label{e-c5}
u_c(t)=e^{A(t)}\int_0^te^{-A(s)}h_e(s)\dif s+e^{-A(t)}\[c+\int_0^te^{A(s)}h_o(s)\dif s\]\nkp\text{for } c\in\bR.\end{equation}
\end{thm}
\begin{proof}  In the case (C5), $b_o=b$ and $a_o=a$. Also, the matrix in \eqref{eoparts} is diagonal
\begin{align}\label{eoparts5}\begin{pmatrix}x_o' \\ x_e'\end{pmatrix} & =\begin{pmatrix}a_o-b_o & 0 \\ 0 & -a_o-b_o\end{pmatrix}\begin{pmatrix}x_o \\ x_e\end{pmatrix}+\begin{pmatrix}h_e \\ h_o\end{pmatrix}.
\end{align}
 and the solutions of \eqref{eoparts5} are given by

\begin{align*}x_o(t) & =e^{A(t)-B(t)}\[\tilde c+\int_0^te^{B(s)-A(s)}h_e(s)\dif s\],\\x_e(t) & =e^{-A(t)-B(t)}\[c+\int_0^te^{A(s)+B(s)}h_o(s)\dif s\],\end{align*}
where $c$, $\tilde c\in\bR$. Since $a$ and $b$ are odd, $A$ and $B$ are even. So, $x_e$ is even independently of the value of $c$. Nevertheless, $x_o$ is odd only when $\tilde c=0$. In such a case, since we need, as in the previous Theorem, that $x_o(T)=0$, we get condition \eqref{conodd2}, which allows us to deduce the first implication of the Theorem.\par
Any solution $u_c$ of \eqref{eq2cp} has the expression \eqref{e-c5}.\par
To show the second implication, it is enough to check that $u$ is a solution of the problem \eqref{eq2cp}.
$$
u_c'(t)=(a(t)-b(t))e^{A(t)-B(t)}\int_0^te^{B(s)-A(s)}h_e(s)\dif s-(a(t)+b(t))e^{-A(t)-B(t)}\[c+\int_0^te^{A(s)+B(s)}h_o(s)\dif s\]+h(t).$$
Now,
\begin{align*}\ & a(t)u_c(-t)+b(t)u_c(t)=a(t)\(-e^{A(t)-B(t)}\int_0^te^{B(s)-A(s)}h_e(s)\dif s+e^{-A(t)-B(t)}\[c+\int_0^te^{A(s)+B(s)}h_o(s)\dif s\]\)\\ & +b(t)\(e^{A(t)-B(t)}\int_0^te^{B(s)-A(s)}h_e(s)\dif s+e^{-A(t)-B(t)}\[c+\int_0^te^{A(s)+B(s)}h_o(s)\dif s\]\)\\ = &-(a(t)-b(t))e^{A(t)-B(t)}\int_0^te^{B(s)-A(s)}h_e(s)\dif s+(a(t)+b(t))e^{-A(t)-B(t)}\[c+\int_0^te^{A(s)+B(s)}h_o(s)\dif s\].
\end{align*}
So clearly,
$$u_c'(t)+a(t)u_c(-t)+b(t)u_c(t)=h(t)\quad\text{for a.e. } t\in I.$$
which ends the proof.
\end{proof}

\section{The mixed case}

When we are not on the cases (C1)-(C5), since the fundamental matrix of $M$ is not given by its exponential matrix, it is more difficult to precise when problem \eqref{eq2cp} has a solution. Here we present some partial results.
\par
Consider the following ODE
\begin{equation}\label{usualp}x'(t)+[a(t)+b(t)]x(t)=0,\quad x(-T)=x(T).\end{equation}
The following lemma gives us the explicit Green's function for this problem. Let $\upsilon=a+b$.
\begin{lem}\label{lem1} Let $h$, $a$ in problem \eqref{usualp} be in $L^1(I)$ and assume $\int_{-T}^{T}\upsilon(t)\dif t\ne0$. Then problem  \eqref{usualp} has a unique solution given by
$$u(t)=\int_{-T}^{T}G_3(t,s)h(t)\dif s,$$
where
\begin{equation}\label{e-G3} G_3(t,s)=\begin{cases}\tau\,e^{\int_t^s\upsilon(r)\dif r}, & s\le t,\\(\tau-1)e^{\int_t^s\upsilon(r)\dif r}, & s> t,\end{cases}\quad \text{and}\quad \tau=\frac{1}{1-e^{-\int_{-T}^T\upsilon(r)\dif r}}.\end{equation}
\end{lem}
\begin{proof}
$$\frac{\partial G_3}{\partial t}(t,s)=\begin{cases}-\tau\,\upsilon(t)\,e^{\int_t^s\upsilon(r)\dif r}, & s\le t,\\-(\tau-1)\upsilon(t)e^{\int_t^s\upsilon(r)\dif r}, & s> t,\end{cases}=-\upsilon(t)G_3(t,s).$$
Therefore,
$$\frac{\partial G_3}{\partial t}(t,s)+\upsilon(t)G_3(t,s)=0,\ s\ne t.$$
Hence,
\begin{align*}
& u'(t)+\upsilon(t)u(t)=\frac{\dif}{\dif t}\int_{-T}^{t}G_3(t,s)h(s)\dif s+\frac{\dif}{\dif t}\int_{t}^{T}G_3(t,s)h(s)\dif s+\upsilon(t)\int_{-T}^{T}G_3(t,s)h(s)\dif s\\= & [G_3(t,t^-)-G_3(t,t^+)]h(t)+\int_{-T}^{T}\[\frac{\partial G_3}{\partial t}(t,s)+\upsilon(t)G_3(t,s)\]h(t)\dif s=h(t)\nkp\text{a.\,e. }t\in I.
\end{align*}
The boundary conditions are also satisfied.
\begin{align*}
& u(T)-u(-T)=\int_{-T}^{T}\[\tau\,e^{\int_T^s\upsilon(r)\dif r}-(\tau-1)e^{\int_{-T}^s\upsilon(r)\dif r}\]h(s)\dif s\\= & \int_{-T}^{T}
\[\frac{e^{\int_T^s\upsilon(r)\dif r}}{1-e^{-\int_{-T}^T\upsilon(r)\dif r}}-\frac{e^{-\int_{-T}^T\upsilon(r)\dif r}\,e^{\int_{-T}^s\upsilon(r)\dif r}}{1-e^{-\int_{-T}^T\upsilon(r)\dif r}}\]h(s)\dif s\\= & \int_{-T}^{T}
\[\frac{e^{\int_T^s\upsilon(r)\dif r}}{1-e^{-\int_{-T}^T\upsilon(r)\dif r}}-\frac{e^{-\int_{T}^s\upsilon(r)\dif r}}{1-e^{-\int_{-T}^T\upsilon(r)\dif r}}\]h(s)\dif s=0.
\end{align*}
\end{proof}
\begin{lem} \begin{equation}
\label{e-Fv}
|G_3(t,s)|\le F(\upsilon):=\frac{e^{\|\upsilon\|_1}}{|e^{\|\upsilon^+\|_1}-e^{\|\upsilon^-\|_1}|}.\end{equation}
\end{lem}
\begin{proof}
Observe that
$$\tau=\frac{1}{1-e^{\|\upsilon^-\|_1-\|\upsilon^+\|_1}}=\frac{e^{\|\upsilon^+\|_1}}{e^{\|\upsilon^+\|_1}-e^{\|\upsilon^-\|_1}}.$$
Hence,
$$\tau-1=\frac{e^{\|\upsilon^-\|_1}}{e^{\|\upsilon^+\|_1}-e^{\|\upsilon^-\|_1}}.$$
On the other hand,
$$e^{\int_t^s\upsilon(r)\dif r}\le\begin{cases} e^{\|\upsilon^-\|_1}, & s\le t,\\e^{\|\upsilon^+\|_1}, & s>t,
\end{cases}$$
which ends the proof.
\end{proof}
The next result proves the existence and uniqueness of solution of $\eqref{eq2cp}$ when $\upsilon$ is `sufficiently small'.
\begin{thm}\label{thmpn}Let $h$, $a$, $b$ in problem \eqref{eq2cp} be in $L^1(I)$ and assume $\int_{-T}^{T}\upsilon(t)\dif t\ne0$. Let $W:=\{(2T)^\frac{1}{p}(\|a\|_{p^*}+\|b\|_{p^*})\}_{p\in[1,+\infty]}$ where $p^{-1}+(p^*)^{-1}=1$. If $F(\upsilon)\|a\|_1(\inf W)<1$, $F(\upsilon)$ defined as in \eqref{e-Fv}, then problem  \eqref{eq2cp} has a unique solution.
\end{thm}
\begin{proof}
With some manipulation we get
 $$h(t)  = x'(t)+a(t)\(\int_t^{-t}x'(s)\dif s+x(t)\)+b(t)x(t)=x'(t)+\upsilon(t)x(t)+a(t)\int_t^{-t}(h(s)-a(s)x(-s)-b(s)x(s))\dif s.$$
Hence,
$$x'(t)+\upsilon(t)x(t)=a(t)\int_t^{-t}(a(s)x(-s)+b(s)x(s))\dif s+a(t)\int_{-t}^{t}h(s)\dif s+h(t).$$
Using $G_{3}$ defined as in \eqref{e-G3} and Lemma \ref{lem1}, it is clear that
$$x(t)=\int_{-T}^TG_3(t,s)a(s)\int_s^{-s}(a(r)x(-r)+b(r)x(r))\dif r\dif s+\int_{-T}^TG_3(t,s)\[a(s)\int_{-s}^{s}h(r)\dif r+h(s)\]\dif s,$$
this is, $x$ is a fixed point of an operator of the form $Hx(t)+\beta(t)$, so, by Banach contraction Theorem, it is enough to prove that $\|H\|<1$ for some compatible norm of $H$.
\par
Using Fubini's Theorem,
$$Hx(t)=-\int_{-T}^{T}\rho(t,r)(a(r)x(-r)+b(r)x(r))\dif r,$$
where $\rho(t,r)=\[\int_{|r|}^{T}-\int_{-T}^{-|r|}\]G_3(t,s)a(s)\dif s$.

If $\int_{-T}^{T}\upsilon(t)\dif t=\|\upsilon^+\|_1-\|\upsilon^-\|_1>0$ then $G_3$ is positive and
$$\rho(t,r)\le\int_{-T}^{T}G_3(t,s)|a(s)|\dif s\le F(\upsilon)\|a\|_1.$$
We have the same estimate for $-\rho(t,r)$.\par
If $\int_{-T}^{T}\upsilon(t)\dif t<0$ we proceed with an analogous argument and arrive as well to the conclusion that
$|\rho(t,s)|<F(\upsilon)\|a\|_1$.\par Hence, $|Hx(t)|\le F(\upsilon)\|a\|_1\int_{-T}^{T}|a(r)x(-r)+b(r)x(r)|\dif r=F(\upsilon)\|a\|_1\|a(r)x(-r)+b(r)x(r)\|_1$.
Thus, it is clear that
$$\|Hx\|_p  \le (2T)^\frac{1}{p}F(\upsilon)\|a\|_1(\|a\|_{p^*}+\|b\|_{p^*})\|x\|_p,\ p\in[1,\infty],$$
which ends the proof.
\end{proof}
\begin{rem} In the hypothesis of Theorem \ref{thmpn}, realize that $F(\upsilon)\ge 1$.
\end{rem}

The following result will let us obtain some information on the sign of the solution of problem \eqref{eq2cp}. In order to prove it, we will use a theorem from \cite{Cab5} we cite below.
\par
Consider an interval $[w,d]\subset I$, the cone
\begin{equation*}\label{eqcone-cs}
K=\{u\in \cC(I): \min_{t \in [w,d]}u(t)\geq c \|u\|\},
\end{equation*}
and the following problem
\begin{equation}\label{eqgenpro2}
x'(t)  =h(t,x(t),x(-t)),\,  t\in I,\quad x(-T)=x(T),
\end{equation}
where $h$ is an $L^1$-Caratheodory function.
Consider the following conditions.

\begin{enumerate}
\item[$(\mathrm{I}_{\protect\rho,\omega}^{1})$] \label{EqB2} There exist $\rho> 0$ and $\omega\in\(0,\frac{\pi}{4T}\]$ such that $f^{-\rho,\rho}_\omega <\omega$ where
$$
  f^{{-\rho},{\rho}}_\omega:=\sup \left\{\frac{h(t,u,v)+\omega v}{\rho }:\;(t,u,v)\in
[ -T,T]\times [ -\rho,\rho ]\times [-\rho,\rho ]\right\}.$$

\item[$(\mathrm{I}_{\protect\rho,\omega}^{0})$] There exists $\rho >0$ such that
$$
f_{(\rho ,{\rho /c})}^\omega\cdot\inf_{t\in [w,d]}\int_{w}^{d}\ol G(t,s)\,ds>1,
$$
where
$$
f_{(\rho ,{\rho /c})}^\omega =\inf \left\{\frac{h(t,u,v)+\omega v}{\rho }%
:\;(t,u,v)\in [w,d]\times [\rho ,\rho /c]\times
[-\rho /c,\rho /c]\right\}.$$
\end{enumerate}

\begin{thm}\textrm{\cite[Theorem 5.15]{Cab5}}\label{thmgen}

Let $\omega\in\(0,\frac{\pi}{2}T\]$. Let $[w,d]\subset I$ such that $w=T-d\in(\max\{0,T-\frac{\pi}{4\omega}\},\frac{T}{2})$. Let
\begin{equation}\label{e-c}c=\frac{[1-\tan(\omega d)][1-\tan(\omega w)]}{[1+\tan(\omega d)][1+\tan(\omega w)]}.\end{equation}

Problem \eqref{eqgenpro2} has at least one non-zero solution
in $K$ if either of the following conditions hold.

\begin{enumerate}

\item[$(S_{1})$] There exist $\rho _{1},\rho _{2}\in (0,\infty )$ with $\rho
_{1}/c<\rho _{2}$ such that $(\mathrm{I}_{\rho _{1},\omega}^{0})$ and $(\mathrm{I}_{\rho _{2},\omega}^{1})$ hold.

\item[$(S_{2})$] There exist $\rho _{1},\rho _{2}\in (0,\infty )$ with $\rho
_{1}<\rho _{2}$ such that $(\mathrm{I}_{\rho _{1},\omega}^{1})$ and $(\mathrm{I}%
_{\rho _{2},\omega}^{0})$ hold.
\end{enumerate}

\end{thm}

\begin{thm}\label{thmmix2}
Let $h\in L^\infty(I)$, $a,b\in L^1(I)$ be such that $0<|b(t)|<a(t)<\omega<\frac{\pi}{2}T$ for a.\,e. $t\in I$ and $\inf h>0$. Then there exists a solution $u$ of \eqref{eq2cp} such that, $u>0$ in $\(\max\{0,T-\frac{\pi}{4\omega}\},\min\{T,\frac{\pi}{4\omega}\}\)$.
\end{thm}
\begin{proof}
Problem \eqref{eq2cp} can be rewritten as
\begin{equation*}\label{eq2cp3} x'(t)=h(t)-b(t)\,x(t)-a(t)\,x(-t),\nkp t\in I,\quad x(-T)=x(T).
\end{equation*}
With this formulation, we can apply Theorem \ref{thmgen}.
Since $0<a(t)-|b(t)|<\omega$ a.\,e., take $\rho_2\in\bR^+$ large enough such that $h(t)<(a(t)-|b(t)|)\rho_2$ a.\,e. Hence, $h(t)<(a(t)-\omega)\rho_2-|b(t)|\rho_2+\rho_2\omega$ for a.\,e. $t\in I$, in particular, $$h(t)<(a(t)-\omega)v-|b(t)|u+\rho_2\omega\le(a(t)-\omega)v+b(t)\,u+\rho_2\omega\text{ for a.\,e. }t\in I;\ u,v\in[-\rho_2,\rho_2].$$ Therefore,
$$\sup \left\{\frac{h(t)-b(t)u-a(t)v+\omega v}{\rho_2 }:\;(t,v)\in
[ -T,T]\times [-\rho_2,\rho_2]\right\}<\omega,$$
and thus, $(\mathrm{I}_{\rho _{2},\omega}^{1})$ is satisfied.\par
Let $[w,d]\subset I$ be such that $[w,d]\subset\(T-\frac{\pi}{4\omega},\frac{\pi}{4\omega}\)$. Let $c$ be defined as in \eqref{e-c}
and $\e=\omega\int_w^d\ol G(t,s)\dif s$.\par Choose $\d\in(0,1)$ such that $h(t)>\[\(1+\frac{c}{\e}\)\omega-(a(t)-|b(t)|)\]\rho_2\d$ a.\,e. and define $\rho_1:=\d c\rho_2$. Therefore, $h>\[(a(t)-\omega)v+b(t)\,u(t)\]\frac{\omega}{\e}\rho_1$ for a.\,e. $t\in I$, $u\in[\rho_1,\frac{\rho_1}{c}]$ and $v\in[-\frac{\rho_1}{c},\frac{\rho_1}{c}]$. Thus,
$$\inf \left\{\frac{h(t)-b(t)u-a(t)v+\omega v}{\rho_1 }:\;(t,v)\in [w,d]\times
[-\rho_1/c,\rho_1/c]\right\}>\frac{\omega}{\e},$$
and hence,  $(\mathrm{I}_{\rho _{1},\omega}^{0})$ is satisfied.
Finally, $(S_1)$ in Theorem \ref{thmgen} is satisfied and we get the desired result.
\end{proof}
\begin{rem} In the hypothesis of Theorem \ref{thmmix2},  if $\omega<\frac{\pi}{4}T$, we can take $[w,d]=[-T,T]$ and continue with the proof of Theorem \ref{thmmix2} as done above. This guarantees that $u$ is positive.
\end{rem}

\section{Appendix: Further considerations}

\subsection{The general case}
The equation $\eqref{proinv1}$, for the case $\phi(t)=-t$, can be reduced to the following system
\begin{align*}\L\begin{pmatrix}x_o' \\ x_e'\end{pmatrix} & =\begin{pmatrix}a_o-b_o & -a_e-b_e \\ a_e-b_e & -a_o-b_o\end{pmatrix}\begin{pmatrix}x_o \\ x_e\end{pmatrix}+\begin{pmatrix}h_e \\ h_o\end{pmatrix},
\end{align*}
where
\begin{align*}
\L=\begin{pmatrix}c_e+d_e & c_o-d_o \\ c_o+d_o & c_e-d_e\end{pmatrix}.
\end{align*}
Hence, if $\det(\L(t))=c(t)c(-t)-d(t)d(-t)\ne 0$ for a.\,e. $t\in I$, $\L(t)$ is invertible a.\,e. and
\begin{align*}\begin{pmatrix}x_o' \\ x_e'\end{pmatrix} & =\L^{-1}\begin{pmatrix}a_o-b_o & -a_e-b_e \\ a_e-b_e & -a_o-b_o\end{pmatrix}\begin{pmatrix}x_o \\ x_e\end{pmatrix}+\L^{-1}\begin{pmatrix}h_e \\ h_o\end{pmatrix}.
\end{align*}
So the general case where $c\not\equiv0$ is reduced to the case studied on Section 3, taking
$$\L^{-1}\begin{pmatrix}a_o-b_o & -a_e-b_e \\ a_e-b_e & -a_o-b_o\end{pmatrix}$$
as coefficient matrix.
\subsection{Computing the matrix exponential}

It is very well known that, in general, it is difficult to compute the exponential of a functional matrix and it is deeply related to the property of the matrix commuting with its integral. Here we summarize the findings of \cite{Kot} on this behalf.
\begin{dfn} Let $S\subset\bR$ be an interval. Define $\cM\subset\cC^1(\bR,\cM_{n\times n}(\bR))$ such that for every $M\in\cM$,
\begin{itemize}
\item there exists $P\in\cC^1(\bR,\cM_{n\times n}(\bR))$ such that $M(t)=P^{-1}(t)J(t)P(t)$ for every $t\in S$ where $P^{-1}(t)J(t)P(t)$ is a Jordan decomposition of $M(t)$;
\item the superdiagonal elements of $J$ are independent of $t$, as well as the dimensions of the Jordan boxes associated to the different eigenvalues of $M$;
\item two different Jordan boxes of $J$ correspond to different eigenvalues;
\item if two eigenvalues of $M$ are ever equal, they are identical in the whole interval $S$.
\end{itemize}
\end{dfn}
It is straightforward to check that the functional matrices appearing in cases (C1)--(C5) belong to $\cM$.
\begin{thm}[\cite{Kot}] Let $M\in\cM$. Then, the following statements are equivalent.
\begin{itemize}
\item $M$ commutes with its derivative.
\item $M$ commutes with its integral.
\item $M$ commutes functionally, that is $M(t)M(s)=M(s)M(t)$ for all $t,s\in S$.
\item $M=\sum_{k=0}^r\gamma_k(t)C^k$ For some $C\in\cM_{n\times n}(\bR)$ and $\gamma_k\in\cC^1(S,\bR)$, $k=1,\dots,r$.
\end{itemize}
Furthermore, any of the last properties imply that $M(t)$ has a set of constant eigenvectors, i.e. a Jordan decomposition $P^{-1}J(t)P$ where $P$ is constant.
\end{thm}
When we first studied the case (C1) --with the other cases we need  similar considerations-- we needed to compute the exponential of the matrix
$$\ol M=\begin{pmatrix}B_e &  -(1+k)A_o\\ (1-k)A_o & -B_e\end{pmatrix}.$$
$\ol M$ has two complex conjugate eigenvalues. What is more, it functionally commutes, so it has a basis of constant eigenvectors given by the constant matrix
$$Y:=\frac{1}{k-1}\begin{pmatrix}i\sqrt{1-k^2} & -i\sqrt{1-k^2} \\ k-1 & k-1 \end{pmatrix}.$$
We have that
$$Y^{-1}\ol M(t)Y=Z(t):=\begin{pmatrix} -B_e-i\,A_o\sqrt{1-k^2} & 0 \\ 0 & -B_e+i\,A_o\sqrt{1-k^2} \end{pmatrix}.$$

Hence,
$$e^{\ol M(t)}=e^{YZ(t)Y^{-1}}=Ye^{Z(t)}Y^{-1}=e^{-B_e(t)}\begin{pmatrix}\cos\(\sqrt{1-k^2}A(t)\) & -\frac{1+k}{\sqrt{1-k^2}}\sin\(\sqrt{1-k^2}A(t)\)\\ \frac{\sqrt{1-k^2}}{1+k}\sin\(\sqrt{1-k^2}A(t)\) & \cos\(\sqrt{1-k^2}A(t)\)\end{pmatrix}.$$

\end{document}